\documentclass[11pt, oneside]{article}

\usepackage{geometry}
\usepackage{amsmath,graphicx,amssymb,amsthm,amstext}
\usepackage{caption}
\usepackage{subcaption}
\usepackage{tikz}          		

\newtheorem{theorem}{Theorem}
\newtheorem{lemma}[theorem]{Lemma}
\newtheorem{corollary}[theorem]{Corollary}
\newtheorem{proposition}[theorem]{Proposition}
\newtheorem{conjecture}[theorem]{Conjecture}

\newenvironment{definition}[1][Definition]{\begin{trivlist}
\item[\hskip \labelsep {\bfseries #1}]}{\end{trivlist}}

\title{Extremal Problems on Generalized Directed Hypergraphs}
\author{Alex Cameron}

\begin{document}
\maketitle

\begin{abstract}
In this paper we define a class of combinatorial structures the instances of which can each be thought of as a model of directed hypergraphs in some way. Each of these models is uniform in that all edges have the same internal structure, and each is simple in that no loops or multiedges are allowed. We generalize the concepts of Tur\'{a}n density, blowup density, and jumps to this class and show that many basic extremal results extend naturally in this new setting. In particular, we show that supersaturation holds, the blowup of a generalized directed hypergraph (GDH) has the same Tur\'{a}n density as the GDH itself, and degenerate GDHs (those with Tur\'{a}n density zero) can be characterized as being contained in a blowup of a single edge. Additionally, we show how the set of jumps from one kind of GDH relates to the set of jumps of another. Since $r$-uniform hypergraphs are an instance of the defined class, then  we are able to derive many particular instances of jumps and nonjumps for GDHs in general based on known results.
\end{abstract}

\section{Introduction}

This paper explores Tur\'{a}n-type problems for a class of relational structures that can each be thought of as generalized directed hypergraphs. This class includes the standard undirected $r$-uniform hypergraphs that have been extensively studied in combinatorics as well as totally directed $r$-uniform hypergraphs where each edge is a set of $r$ vertices under a linear ordering. Instances of this latter structure have been studied in the extremal setting by Erd\H{o}s, Brown, Simonovits, Harary, and others \cite{brown1973, brown1969, brown2002, brown1984}. Other instances of this class are uniform versions of the model used to represent definite Horn formulas in the study of propositional logic and knowledge representation \cite{angluin1992, russell2002}. The combinatorial properties of this model have been recently studied by Langlois, Mubayi, Sloan, and Gy. Tur\'{a}n in \cite{langlois2009} and by this author in \cite{cameron2015} and \cite{cameron2015deg}. Other structures in this class are slight variations on the $d$-simplex structures studied by Leader and Tan in \cite{leader2010}.

Tur\'{a}n-type extremal problems for uniform hypergraphs make up a large and well-known area of research in combinatorics that ask the following: ``Given a family of forbidden $r$-uniform hypergraphs $\mathcal{F}$ what is the maximum number of edges an $r$-uniform hypergraph on $n$ vertices can have without containing any member of $\mathcal{F}$ as a (not necessarily induced) subgraph?" Such problems were named after Paul Tur\'{a}n due to his important early results and conjectures concerning forbidden complete $r$-graphs \cite{turan1941, turan1954, turan1961}.

A related question for undirected hypergraphs was proposed by Erd\H{o}s known as the jumping constant conjecture. A real number $\alpha \in [0,1)$ is called a jump for an integer $r \geq 2$ if there exists some positive constant $c$ which depends only on $\alpha$ such that for any $\epsilon > 0$ and positive integer $l$ there exists a positive integer $N$ for which any $r$-uniform hypergraph on $n \geq N$ vertices which has edge density at least $\alpha + \epsilon$ contains a subgraph on $l$ vertices with edge density at least $\alpha + c$. It is well-known that when $r=2$, every $\alpha \in [0,1)$ is a jump \cite{erdos1966, erdos1946}. Moreover, every $\alpha \in \left[0, \frac{r!}{r^r}\right)$ is a jump for $r \geq 3$ \cite{erdos1971}. In 1984, Frankl and R\"{o}dl disproved the jumping constant conjecture when they found the first instance of a nonjump for each $r \geq 3$\cite{frankl1984}. Since then many infinite sequences of nonjumps have been found, but the smallest known nonjump to date is $\frac{5r!}{2r^r}$ for each $r \geq 3$ determined by Frankl, Peng, R\"{o}dl, and Talbot in \cite{frankl2007}. The only additional jumps that have been found are all $\alpha \in [0.2299,0.2316), \left[0.2871,\frac{8}{27} \right)$ for $r=3$ found by Baber and Talbot in \cite{baber2011} using Razborov's flag algebra method \cite{razborov2007}.

Extremal issues of these kinds have also been considered for digraphs and multigraphs (with bounded multiplicity) in \cite{brown1973} and \cite{brown1969} and for the more general directed multi-hypergraphs in \cite{brown1984}. In \cite{brown1969}, Brown and Harary determined the extremal numbers for several types of specific directed graphs including all tournaments - that is, a digraph with one edge in some orientation between every pair of vertices. In \cite{brown1973}, Brown, Erd\H{o}s, and Simonovits determined the general structure of extremal sequences for every forbidden family of digraphs analogous to the Tur\'{a}n graphs for simple graphs.

In \cite{brown1984}, Brown and Simonovits proved several general extremal results about $r$-uniform directed $q$-hypergraphs. In this model the edges are ordered $r$-tuples of vertices with multiplicity up to $q$ for some fixed positive integer $q$. Among their results on this model are three that will be reproduced in this paper in a more general setting: Supersaturation, Continuity, and Approximation. Roughly speaking supersaturation implies that a large graph with an edge-density more than the Tur\'{a}n density for a particular forbidden family must contain many copies of members of that family. Continuity shows that given an infinite forbidden family, we can get arbitrarily close to its extremal number with a finite subfamily. Approximation is a structural result that shows that given a forbidden family, we can approximate an extremal sequence to an arbitrarily small difference by taking some sequence of graphs that all exclude this family and which all fall into some ``nice" form. All of these notions will be made rigorous in the paper.

In \cite{langlois2010} and \cite{langlois2009}, Langlois, Mubayi, Sloan, and Gy. Tur\'{a}n studied extremal properties of certain small configurations in a directed hypergraph model. This model can be thought of as a $2 \rightarrow 1$ directed hypergraph where each edge has three verticies, two of which are ``tails" and the third is a ``head." They determined the extremal number for one such subgraph with two edges, and found the extremal number of a second configuration with two edges up to asymptotic equivalence. In \cite{cameron2015} and \cite{cameron2015deg}, this author followed up this work and found the exact extremal numbers for every $2 \rightarrow 1$ directed hypergraph with exactly two edges. This $2 \rightarrow 1$ model is one instance of the class of models discussed in this paper.

The graph theoretic properties of a more general definition of a directed hypergraph were studied by Gallo, Longo, Pallottino, and Nguyen in \cite{gallo1993}. There a directed hyperedge was defined to be some subset of vertices with a partition into head vertices and tail vertices. This is a nonuniform version of models considered in this paper.

The totally directed hypergraph model considered in \cite{brown1984} and the $r \rightarrow 1$ directed hypergraph model resulting from the study of Horn clauses both lead to the natural question of all possible ways to define a directed hypergraph. The definition in this paper of the class of general directed hypergraph models attempts to unify all of the possible ``natural" ways one could define a directed hypergraph so that certain extremal questions can be answered about all of them at once. Adding to the motivation of considering more general structures is the recent interest in Razborov's flag algebra method which applies to all relational theories and not just undirected hypergraphs. The fact that the $d$-simplex model studied by Leader as well as many other somewhat geometric models come out of the class defined in this paper was a very interesting accident.

This paper is organized as follows. In Section 2, we define the class of generalized directed hypergraphs and extend the concepts of Tur\'{a}n density, blowups, and supersaturation to this setting. In Section 3, we define the idea of a jump for a given model of directed hypergraphs and prove several results about these jumps and how the jumps from one instance of the class relate to jumps in another. In Section 4, we adapt a couple of results proved in \cite{brown1984} for totally directed hypergraphs with multiplicity to any GDH. In Section 5, we ask some questions that arose from studying these structures and discuss alternate definitions that would generalize the concept further.

\section{Basic Definitions and Results}

The following definition for a generalized directed hypergraph is intended to include most uniform models that could reasonably be called uniform directed hypergraphs. This includes models where the edges are $r$-sets each under some partition into $k$ parts of fixed sizes $r_1,\ldots,r_k$ with some linear ordering on the $k$ parts. The definition only includes structures where an $r$-set could include multiple edges up to the number of possible orientations allowed. That is, we do not consider the ``oriented" versions of the models where only one edge is allowed per $r$-set. The definition is given in terms of logic and model theory for convenience only. No deep results from those subfields are used. The use of this notation also makes further generalizations like nonuniform directed hypergraphs or oriented directed hypergraphs easy.

\begin{definition}
\label{GDHdef}
Let $\mathcal{L}=\{E\}$, a language with one $r$-ary relation symbol $E$. Let $T$ be an $\mathcal{L}$-theory that consists of a single sentence of the form \[\forall x_1 \cdots x_r E(x_1,\ldots,x_r) \implies \bigwedge_{i \neq j} x_i \neq x_j \land \bigwedge_{\pi \in J_T} E(x_{\pi(1)},\ldots,x_{\pi(r)})\] for some subgroup of the group of permutations on $r$ elements, $J_T \subseteq S_r$. Call such a theory a \emph{generalized directed hypergraph theory} and any finite model of $T$ is a \emph{generalized directed hypergraph (GDH)}.
\end{definition}

Note that this definition includes graphs, hypergraphs, and $r \rightarrow 1$ directed hypergraphs. For example, the theory for a $2 \rightarrow 1$ directed hypergraph is \[T = \{\forall xyz E(x,y,z) \implies x \neq y \land x \neq z \land y \neq z \land E(y,x,z)\}.\] It is easy to see that when $r=2$ we have only two GDH theories. The theory associated with the group $S_2$ is the theory of graphs, and the theory associated with the trivial group is the theory of directed graphs.

When $r=3$ there are six subgroups of $S_3$. Three of these are all isomorphic to $\mathbb{Z}_2$ with each generated by a permutation that swaps two elements. The corresponding GDH theory for any of these can be thought of as having pointed $3$-sets for edges or as being ($2 \rightarrow 1$)-graphs. Of the other subgroups, $S_3$ itself gives the theory of undirected $3$-uniform hypergraphs, the trivial group gives totally directed $3$-edges, and the subgroup generated by a three-cycle isomorphic to $\mathbb{Z}_3$ yields a GDH theory where the edges can be thought of as $3$-sets that have some kind of cyclic orientation - either clockwise or counter-clockwise. Figure~\ref{lattice} summarizes the models of GDHs when $r=3$. Note that in general, $S_r$ always corresponds to the normal undirected $r$-graph model and the trivial group always corresponds to totally directed hypergraphs.

\begin{figure}
  \centering
      \begin{tikzpicture}
      		\node[above] at (0,3) {$S_3$};
		\node[left] at (-1,2) {$\mathbb{Z}_3 $};
		\node[right] at (1,1) {$\mathbb{Z}_2$};
		\node[below] at (0,0) {$<i>$};
		\draw[->] (-1,2) -- (0,3);
		\draw[->] (1,1) -- (0,3);
		\draw[->] (0,0) -- (-1,2);
		\draw[->] (0,0) -- (1,1);
		
		\node at (2.5, 1.5) {$\implies$};
		
		\filldraw [black] (5.5,3.25) circle (1pt);
		\filldraw [black] (6,3.25) circle (1pt);
		\filldraw [black] (6.5,3.25) circle (1pt);
		\draw[thick] (5.5,3.25) -- (6.5,3.25);
		
		\filldraw [black] (4.25,2.25) circle (1pt);
		\filldraw [black] (4.75,1.75) circle (1pt);
		\filldraw [black] (3.75,1.75) circle (1pt);
		\draw[thick,->] (4.25,2.25) -- (4.75,1.75);
		\draw[thick,->] (4.75,1.75) -- (3.75,1.75);
		\draw[thick,->] (3.75,1.75) -- (4.25,2.25);
		
		\filldraw [black] (7.25,1.25) circle (1pt);
		\filldraw [black] (7.25,0.75) circle (1pt);
		\filldraw [black] (7.75,1) circle (1pt);
		\draw[thick] (7.25,1.25) -- (7.25,0.75);
		\draw[thick,->] (7.25,1) -- (7.75,1);
		
		\filldraw [black] (5.5,-0.25) circle (1pt);
		\filldraw [black] (6,-0.25) circle (1pt);
		\filldraw [black] (6.5,-0.25) circle (1pt);
		\draw[thick, ->] (5.5,-0.25) -- (6.5,-0.25);
		
		\draw[->] (6,0) -- (5,2);
		\draw[->] (6,0) -- (7,1);
		\draw[->] (7,1) -- (6,3);
		\draw[->] (5,2) -- (6,3);
	\end{tikzpicture}
  \caption{The subgroup lattice of $S_3$ and the corresponding lattice of directed hypergraphs.}
  \label{lattice}
\end{figure}
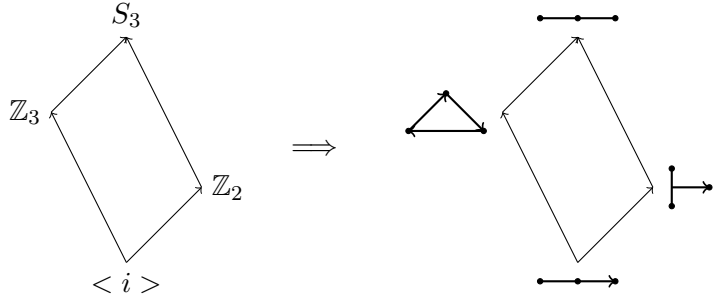

A fun thought experiment is to consider the kinds of edges that arise when $r=4$. Many of them are geometric in nature. For instance, the alternating group $A_4$ gives a theory where edges can be thought of tetrahedrons (at least in an abstract sense). In fact, in \cite{leader2010} Leader and Tan  study the ``oriented" versions of the models that come from the alternating groups for any $r \geq 3$.

In this paper when the theory is not specified we are simply discussing GDHs that are all models of the same fixed theory. When discussing multiple theories we will often refer to $T$-graphs to mean models of a GDH theory $T$. Throughout the paper, $J_T$ will always stand for the subgroup $J_T \subseteq S_r$ that determines the GDH theory $T$ and $m_T$ will always be the order of this subgroup, $m_T = |J_T|$. Also, $V_G$ and $E_G$ will be used to denote the underlying set of elements of a model $G$ and its relation set respectively.

The following basic propositions are given without proof. The first is a simple consequence thatwe are working in a relational language, and the second results from the fact that $J_T$ is a group.

\begin{proposition}
For any GDH theory $T$ and any nonnegative integer $n$, there exists a GDH $G \models T$ on $n$ elements. Moreover, for any nonnegative integer $k<n$, the substructure of $G$ induced on any $k$-subset of the elements of $G$ is also a $T$-graph.
\end{proposition}

\begin{proposition}
Given a GDH $G$ with $r$-ary relation set $E_G$, there exists an equivalence relation $\sim$ on $E_G$ defined by \[(a_1,\ldots,a_r) \sim (b_1,\ldots,b_r)\] if and only if for each $i$, $b_i = a_{\pi(i)}$ for some $\pi \in J_T$.
\end{proposition}

We can now use these propositions to extend the concepts of extremal graph theory to GDHs in a natural way.

\begin{definition}
For any GDH $G$, an \emph{edge} of $G$ will always refer to an equivalence class of $[E_G]_{\sim}$.
\end{definition}

\begin{definition}
Given a GDH $G$ on $n$ elements, denote the number of edges of $G$ by $e_T(G)$ and let the \emph{edge density} of $G$ be defined as \[d_T(G) := \frac{e_T(G)}{\frac{r!}{m_T} {n \choose r}}.\]
\end{definition}

Note that since \[e_T(G) = \frac{|E_G|}{m_T},\] then the density is \[d_T(G) = \frac{(n-r)!|E_G|}{n!}\] and could have been defined this way while mostly avoiding talk of edges as equivalence classes of $E_G$. However, the above definition makes the following extremal concepts reduce to their standard definitions in the undirected case.

\begin{definition}
Given two GDHs $G$ and $H$ and a function $\psi: V_H \rightarrow V_G$, we say that $\psi$ is a \emph{homomorphism} if for all $(a_1,\ldots,a_r) \in E_H$, $(\psi(a_1),\ldots,\psi(a_r)) \in E_G$.

We say that $G$ contains a copy of $H$ if there exists some injective homomorphism, $\psi:V_H \rightarrow V_G$. Otherwise, we say that $G$ is $H$-free. Similarly, we would say that a GDH $G$ is $\mathcal{F}$-free for some family $\mathcal{F}$ of GDHs if $G$ is $F$-free for all $F \in \mathcal{F}$.
\end{definition}

\begin{definition}
Given a family of GDHs $\mathcal{F}$ and a positive integer $n$, let the $n$th \emph{extremal number}, $\text{ex}_T(n,\mathcal{F})$, be defined as the maximum number of edges over all $\mathcal{F}$-free GDHs on $n$ elements, \[\text{ex}_T(n,\mathcal{F}) := \max_{\mathcal{F}\text{-free } G_n} \{e_T(G_n) \}.\] The \emph{Tur\'{a}n density} of $\mathcal{F}$ is defined as \[\pi_T(\mathcal{F}) := \lim_{n \rightarrow \infty} \frac{\text{ex}_T(n,\mathcal{F})}{\frac{r!}{m_T} {n \choose r}}.\]
\end{definition}

Our first main result is to show that these Tur\'{a}n densities exist for any GDH theory. The proof is the standard averaging argument used to show that these limiting densities exist for families of undirected hypergraphs \cite{keevash2011}.

\begin{theorem}
For any GDH family $\mathcal{F}$ the Tur\'{a}n density exists.
\end{theorem}

\begin{proof}
Let $G$ be an $\mathcal{F}$-free GDH on $n$ elements with $\text{ex}_T(n,\mathcal{F})$ edges. For each $i=1,\ldots,n$ let $G^i$ be the subGDH of $G$ induced by removing the $i$th vertex. Each edge of $G$ appears in exactly $n-r$ of these subGDHs. Therefore, \[(n-r)e_T(G) = e_T(G^1)+ \cdots e_T(G^n).\] Moreover, $e_T(G) = \text{ex}_T(n,\mathcal{F})$ and each $G^i$ is also $\mathcal{F}$-free so $e_T(G^i) \leq \text{ex}_T(n-1,\mathcal{F})$. Therefore, \[\text{ex}_T(n,G) \leq \frac{n}{n-r} \text{ex}_T(n-1,\mathcal{F}).\] So \[\frac{\text{ex}_T(n,G)}{\frac{r!}{m_T} {n \choose r}} \leq \frac{n}{n-r} \frac{\text{ex}_T(n-1,\mathcal{F})}{\frac{r!}{m_T} {n \choose r}} = \frac{\text{ex}_T(n-1,\mathcal{F})}{\frac{r!}{m_T} {n-1 \choose r}}.\] Therefore, the sequence of these extremal densities is monotone decreasing as a function of $n$ in the range $[0,1]$. Hence, the limit exists.
\end{proof}

\subsection{Blowups and Blowup Density}

We'll now extend the concept of the blowup of uniform hypergraphs to the more general setting of GDHs and define the corresponding notion of the blowup density. As with hypergraphs, the blowup of a GDH can be thought of as the replacement of each vertex with many copies and taking all of the resulting edges. Formally,

\begin{definition}
Let $G$ be a GDH with $V_G = \{x_1,\ldots,x_n\}$, and let $t = (t_1,\ldots, t_n)$ be a tuple of positive integers. Define the \emph{$t$-blowup} of $G$ to be the $\mathcal{L}$-structure $G(t)$ where \[V_{G(t)} = \{x_{11},\ldots,x_{1t_1},\ldots,x_{n1},\ldots,x_{nt_n}\}\] and \[(x_{i_1j_1},\ldots,x_{i_rj_r}) \in E_{G(t)} \iff (x_{i_1},\ldots,x_{i_r}) \in E_G.\]
\end{definition}

\begin{proposition}
Let $G$ be a GDH on $n$ vertices, and let $t=(t_1,\ldots,t_n)$ be a tuple of positive integers. Then the $t$-blowup of $G$ is also a GDH.
\end{proposition}

\begin{proof}
We need only show that the $\mathcal{L}$-structure $G(t)$ models $T$. So let \[(x_{i_1j_1},\ldots,x_{i_rj_r}) \in E_{G(t)}.\] Then $(x_{i_1},\ldots,x_{i_r}) \in E_G$. Since $G \models T$ this implies that $i_a \neq i_b$ whenever $a \neq b$. Hence, the elements $x_{i_aj_a} \neq x_{i_bj_b}$ whenever $a \neq b$. It also implies that $(x_{i_{\pi(1)}},\ldots,x_{i_{\pi(r)}}) \in E_G$ for any $\pi \in J_T$. Hence, \[(x_{i_{\pi(1)}j_{\pi(1)}},\ldots,x_{i_{\pi(r)}j_{\pi(r)}}) \in E_{G(t)}\] for any $\pi \in J_T$. Therefore, $G(t) \models T$.
\end{proof}

Next, we consider the edge density of a given blowup by defining the edge polynomial for a GDH. 

\begin{definition}
Let $G$ be a GDH on $n$ vertices. For each $r$-set $R \in {V_G \choose r}$, let $e_R$ be the number of edges of $G$ in $R$. Then let the \emph{edge polynomial} be \[p_G(x) := \sum_{R \in {V_G \choose r}} e_R\prod_{i \in R} x_i.\]
\end{definition}

This polynomial is a simple generalization of the standard edge polynomial for undirected hypergraphs. To see this more easily note that for a given GDH $G$, the edges of $G$ are in bijection with the monomials the sum $p_G$ were we to write the sum out with no coefficients greater than one.

From this we see that the edge density of the $(t_1,\ldots,t_n)$-blowup of $G$ is \[\frac{p_G(t_1,\ldots,t_n)}{\frac{r!}{m_T}{t \choose r}} = m_T\frac{p_G(t_1,\ldots,t_n)}{t(t-1)\cdots(t-r+1)}\] where $t= \sum t_i$. Let $t$ increase to infinity and for each $t$ pick a vector $(t_1,\ldots,t_n)$ that maximizes this edge density. Then this sequence of densities is asymptotically equivalent to the sequence of numbers \[m_T p_G\left(\frac{t_1}{t},\ldots,\frac{t_n}{t}\right).\] This motivates the following definition.

\begin{definition}
Let $G$ be a GDH on $n$ vertices. Let \[S^n = \left\{(x_1,\ldots,x_n)|x_i \geq 0 \land \sum_{i=1}^nx_i = 1\right\},\] the standard $(n-1)$-dimensional simplex. Define the \emph{blowup density} of $G$ as \[b_T(G) = m_T \max_{x \in S^n} \{p_G(x)\}.\]
\end{definition}

Since any $x \in S^n$ is the limit of some sequence $\left\{\left(\frac{t_1}{t},\ldots,\frac{t_n}{t}\right)\right\}$ with positive $t_i$ as $t \rightarrow \infty$, then the blowup density of a GDH $G$ is the best limiting density of any sequence of blowups of $G$.

The remaining definition and basic result about blowups given in this subsection will be useful when extending results about jumps and nonjumps from undirected hypergraphs to GDHs generally in Section 3.

\begin{definition}
Let $T'$ and $T$ be GDH theories such that $J_{T'} \subseteq J_{T} \subseteq S_r$. For a $T$-graph $F$ and a $T'$-graph $F'$ we say that $F$ \emph{contains} $F'$ if $V_F = V_{F'}$ and every edge of $F'$ is contained in some edge of $F$ (where the edges are considered under their equivalence class definition as subsets of $E_F$ and $E_{F'}$). We say that $F$ is the \emph{minimum $T$-container} of $F'$ if $F$ has no edges that do not contain edges of $F'$.
\end{definition}

\begin{proposition}
\label{containing}
Let $T'$ and $T$ be GDH theories such that $J_{T'} \subseteq J_{T} \subseteq S_r$. Let $F'$ be a $T'$-graph and let $F$ be the minimum $T$-container of $F'$. Then \[\frac{m_{T'}}{m_T}b_T(F) \leq b_{T'}(F') \leq b_T(F)\] with equality on the left if $F'$ has exactly one edge contained in each edge of $F$ and equality on the right if each edge of $F$ contains all $\frac{m_T}{m_{T'}}$ possible edges of $F'$.

Moreover, if $F'$ has exactly $k$ edges contained in each edge of $F$, then \[b_{T'}(F') =\frac{km_{T'}}{m_T}b_T(F).\]
\end{proposition}

\begin{proof}
Let $|V_{F'}| = |V_{F}|=v$, then for any $x \in S^v$, \[p_F(x) \leq p_{F'}(x) \leq \frac{m_T}{m_{T'}}p_F(x)\] with equality on the left if $F'$ has exactly one edge contained in each edge of $F$ and equality on the right if each edge of $F$ contains all $\frac{m_T}{m_{T'}}$ possible edges of $F'$. Hence, \[\max_{x \in S^v}{p_F(x)} \leq \max_{x \in S^v}{p_{F'}(x)} \leq \max_{x \in S^v}{\frac{m_T}{m_{T'}}p_F(x)}.\] This implies that \[\frac{m_{T'}}{m_T}b_T(F) \leq b_{T'}(F') \leq b_T(F).\] In particular, if $F'$ has exactly $k$ edges contained in each edge of $F$, then for any $x \in S^v$, \[p_{F'}(x) = kp_F(x)\] which implies the result.
\end{proof}

\subsection{Supersaturation and Related Results}

Supersaturation holds for GDHs as it does for undirected hypergraphs, and the proof of this result is the same as the one for hypergraphs found in \cite{keevash2011} with only minor differences.

\begin{theorem}[Supersaturation]
\label{supersaturation}
Let $F$ be a GDH on $k$ elements. Let $\epsilon > 0$. For sufficiently large $n \geq n_0(F,\epsilon)$, any GDH $G$ on $n$ elements with density $d(G) \geq \pi_T(F) + \epsilon$ will contain at least $c {n \choose k}$ copies of $F$ for some constant $c = c(F,\epsilon)$.
\end{theorem}

\begin{proof}
Fix some positive integer $l$ so that \[\text{ex}_T(l,F) < \left(\pi_T(F) + \frac{\epsilon}{2}\right) \frac{r!}{m_T} {l \choose r}.\] Let $G$ be a GDH on $n > l$ elements with edge density $d_T(G) \geq \pi(F) + \epsilon$. Then $G$ must contain more than $\frac{\epsilon}{2}{n \choose l}$ $l$-sets with density at least $\pi_T(F) + \frac{\epsilon}{2}$. Otherwise, at most $\frac{\epsilon}{2}{n \choose l}$ $l$-sets contain more than $\left(\pi_T(F) + \frac{\epsilon}{2}\right){l \choose r}$ edges. Therefore, we can count the number of edges in $G$ by $l$-sets and get an upper bound of \[{n-r \choose l-r}e_T(G) \leq \frac{\epsilon}{2} {n \choose l} {l \choose r} \frac{r!}{m_T} + \left(1-\frac{\epsilon}{2}\right){n \choose l}\left(\pi_T(F)+\frac{\epsilon}{2}\right){l \choose r}\frac{r!}{m_T}.\] We can now replace $e_T(G)$ since \[e_T(G) \geq \left(\pi_T(F)+\epsilon\right){n \choose r} \frac{r!}{m_T}.\] This is enough to get the contradiction.

Since $G$ contains more than $\frac{\epsilon}{2}{n \choose l}$ $l$-sets with density at least $\pi_T(F) + \frac{\epsilon}{2}$, then it contains a copy of $F$ in each. A given copy of $F$ appears in ${n-k \choose l-k}$ $l$-sets of $G$. Therefore, there are more than \[\frac{\epsilon}{2}{n \choose l}{n-k \choose l-k}^{-1} = c {n \choose k}\] distinct copies of $F$ in $G$ where \[c = \frac{\epsilon}{2}{l \choose k}^{-1}.\]
\end{proof}

Similarly, the following theorem is an extension from the same result for undirected hypergraphs, and the proof is an adaptation of the one found in \cite{keevash2011}.

\begin{theorem}
\label{blowup}
Let $F$ be a GDH on $k$ vertices and let $t=(t_1,\ldots,t_k)$ be an $k$-tuple of positive integers. Then $\pi_T(F) = \pi_T(F(t))$.
\end{theorem}

\begin{proof}
That $\pi_T(F) \leq \pi_T(F(t))$ is trivial since $F(t)$ contains a copy of $F$ so any $F$-free GDH is automatically $F(t)$-free.

Therefore, we only need to show that $\pi_T(F) \geq \pi_T(F(t))$. Suppose not, then for sufficiently large $n$ there exists some $F(t)$-free GDH $G$ on $n$ elements with edge density strictly greater than $\pi_T(F)$. By supersaturation this implies that $G$ contains $c{n \choose k}$ copies of $F$.

Define $G^*$ to be the $k$-uniform hypergraph where $V_{G^*} = V_{G}$ and $\{a_1,\ldots,a_k\} \in E_{G^*}$ iff and only if $\{a_1,\ldots,a_k\}$ contains a copy of $F$ in $G$. Since the edge density of $G^*$ is $c>0$, then for large enough $n$, $G^*$ must contain an arbitrarily large complete $k$-partite subgraph.

For each edge $F$ maps to the vertices in at least one out of $k!$ total possible ways to make an injective homomorphism in $G$. Therefore, by Ramsey Theory, if we take the parts of this complete $k$-partite subgraph large enough and color the edges by the finite number of non-isomorphic ways that $F$ could possibly map to the $k$ vertices, we will get an arbitrarily large monochromatic $k$-partite subgraph where each part has $t$ vertices. This must have been a copy of $F(t)$ in $G$, a contradiction.
\end{proof}

The fact that the Tur\'{a}n density of a blowup equals the Tur\'{a}n density of the original GDH leads to the following nice characterization of degenerate families of GDH - those families with Tur\'{a}n density zero.

\begin{theorem}[Characterization of Degenerate GDH]
\label{degenerate}
Let $\mathcal{F}$ be some family of GDHs, then $\pi_T(\mathcal{F}) = 0$ if and only if some member $F \in \mathcal{F}$ is a subGDH of the $t$-blowup of a single edge for some vector, $t=(t_1,\ldots,t_r)$, of positive integers. Otherwise, $\pi(\mathcal{F}) \geq \frac{m_T}{r^r}$.
\end{theorem}

\begin{proof}
Suppose that no member of $\mathcal{F}$ is such a blowup. Then no member is contained in the $(t,t,\ldots,t)$-blowup of $S$. Let $S(t)$ stand for this blowup, then the sequence of GDHs, $\{S(t)\}_{t=1}^{\infty}$, is an $\mathcal{F}$-free sequence. The density of any such $S(t)$ is \[d_T(S(t)) = \frac{t^r}{\frac{r!}{m_T} {tr \choose r}} = \frac{m_T t^r (tr-r)!}{(tr)!}.\] These densities tend to $\frac{m_T}{r^r}$ as $t$ increases. Therefore, \[\pi_T(\mathcal{F}) \geq  \frac{m_T}{r^r} > 0.\]

Conversely, suppose some $F \in \mathcal{F}$ is a $(t_1,\ldots,t_r)$-blowup of a single edge. By Theorem~\ref{blowup}, $\pi_T(F) = \pi_T(S) = 0$ since $\text{ex}_T(n,S)$=0 for all $n$. Therefore, $\pi_T(\mathcal{F})=0$.
\end{proof}

\section{Jumps}

Now we turn to the issue of finding jumps and nonjumps for GDH theories. The definition of a jump for undirected hypergraphs extends naturally to this setting as does the important connection between jumps and blowup densities.

\begin{definition}
Let $T$ be a GDH theory, then $\alpha \in [0,1)$ is a \emph{jump} for $T$ if there exists a $c >0$ such that for any $\epsilon > 0$ and any positive integer $l$, there exists a positive integer $n_0(\alpha, \epsilon, l)$ such that any GDH $G$ on $n \geq n_0$ elements that has at least $(\alpha + \epsilon) \frac{r!}{m_T} {n \choose r}$ edges contains a subGDH on $l$ elements with at least $(\alpha + c) \frac{r!}{m_T} {l \choose r}$ edges.
\end{definition}

Note that by Theorem~\ref{degenerate} every $\alpha \in \left[0,\frac{m_T}{r^r}\right)$ is a jump for any $r$-ary GDH theory $T$. This generalizes the well-known result of Erd\H{o}s \cite{erdos1964} that every $\alpha \in [0,\frac{r!}{r^r})$ is a jump for $r$-graphs. The following important theorem on jumps for GDH theories was originally shown by Frankl and R\"{o}dl \cite{frankl1984} for undirected hypergraphs. Their proof works equally well in this setting so the differences here are in name only.

\begin{theorem}
\label{mainguy}
The GDH theory $T$ has a jump $\alpha$ if and only if there exists a finite family $\mathcal{F}$ of GDHs such that $\pi_T(\mathcal{F}) \leq \alpha$ and $b_T(F) > \alpha$ for each $F \in \mathcal{F}$.
\end{theorem}

\begin{proof}
Let $\alpha$ be a jump and let $c$ be the supremum of all corresponding ``lengths" $c$ to the jump. Fix a positive integer $k$ so that \[{k \choose r} \left(\alpha + \frac{c}{2}\right) > \alpha \frac{k^r}{r!}.\] Let $\mathcal{F}$ be the family of all GDHs on $k$ elements with at least $\left(\alpha + \frac{c}{2}\right) {k \choose r} \frac{r!}{m_T}$ edges. Then $\pi_T(\mathcal{F}) \leq \alpha$ since any slightly larger density implies arbitrarily large subsets with density $\alpha + c$. This in turn would imply the existence of a $k$-subset with density at least $\alpha + c$. This $k$-subset would include some member of $\mathcal{F}$. On the other hand, a given $F \in \mathcal{F}$ will have blowup density \[b_T(F) \geq m_T p_F\left(\frac{1}{k},\ldots,\frac{1}{k}\right)>\alpha.\]

Conversely, suppose that such a finite family $\mathcal{F}=\{F_1,\ldots,F_k\}$ exists. Let $\epsilon >0$ and let $\{G_n\}$ be an infinite sequence of GDHs with density that tends to $\alpha + \epsilon$. As in the proof of Theorem~\ref{supersaturation}, for any positive integer $l$, $G_n$ must contain at least $\frac{\epsilon}{2}{n \choose l}$ $l$-subsets with density at least $\alpha + \frac{\epsilon}{2}$.

Let $l$ be large enough so that any GDH on $l$ vertices with density at least $\alpha + \frac{\epsilon}{2}$ contains some $F_i$ from $\mathcal{F}$. Therefore, any $G_n$ with $n > l$ contains $\frac{\epsilon}{2}{n \choose l}$ $l$-sets each with some $F_i$. Since there are only $k$ members of $\mathcal{F}$, then this implies that at least $\frac{\epsilon}{2k}{n \choose l}$ $l$-sets contain the same $F_i$.

Let $|V(F_i)|=v_i$. By the proof of Theorem~\ref{supersaturation} this implies that there is some positive constant $b$ such that $G_n$ contains at least $b{n \choose v_i}$ distinct copies of $F_i$. By the proof of Theorem~\ref{blowup} this shows that if $n$ is large enough, then we get a copy of an arbitrarily large $t$-blowup of $F_i$.

Let $c = \min_{F_i \in \mathcal{F}}{b_T(F_i)}$. For some subset $\mathcal{F}' \subseteq \mathcal{F}$, each $F_i \in \mathcal{F}'$ yields an infinite subsequence of $\{G_n\}$ which contains arbitrarily large $t$-blowups of $F_i$. The densities of these blowups all tend to at least $c$. Therefore, for any positive integer $m$, there exists an $m$-set of each $\{G_n\}$ for sufficiently large $n$ with density at least $\alpha + c$. Hence, $\alpha$ is a jump.
\end{proof}

The following proposition is needed to compare jumps between different GDH theories.

\begin{proposition}
\label{thmB}
The GDH theory $T$ has a jump $\alpha$ if and only if there exists some $c > 0$ such that for all families $\mathcal{F}$ of GDHs, either $\pi_T(\mathcal{F}) \leq \alpha$ or $\pi_T(\mathcal{F}) \geq \alpha + c$.
\end{proposition}

\begin{proof}
Let $\alpha$ be a jump for $T$ and let $c>0$ be some corresponding ``length" to the jump. Suppose that $\mathcal{F}$ is a finite family of GDHs of type $T$ for which $\alpha < \pi_T(\mathcal{F}) < \alpha + c$. Let $\{G_n\}$ be a sequence of extremal $\mathcal{F}$-free GDHs. For each positive integer $k$ there exists some $G_n$ that contains a $k$-subset with at least $(\alpha + c){k \choose r} \frac{r!}{m_T}$ edges. Take the sequence of these subsets. They are all $\mathcal{F}$-free by assumption, and the limit of their densities is at least $\alpha + c$. Therefore, $\pi_T(\mathcal{F}) \geq \alpha + c$, a contradiction.

Conversely, assume that $\alpha$ is not a jump. Let $c>0$, then for some $0<\epsilon<c$ and some positive integer $l$, there exists an infinite sequence of GDHs, $\{G_n\}$ for which each GDH has density at least $\alpha + \epsilon$ and all $l$-sets have strictly less than $(\epsilon + c){l \choose r}\frac{r!}{m_T}$ edges. Hence, $\{G_n\}$ is $\mathcal{F}$-free where $\mathcal{F}$ is the set of all $l$-GHDs with at least $(\alpha + c){l \choose r} \frac{r!}{m_T}$ edges. So $\pi_T(\mathcal{F}) \geq \alpha + \epsilon$. Since any GDH with density at least $\alpha + c$ must have an $l$-set with density at least $\alpha + c$, then $\pi_T(\mathcal{F}) < \alpha + c$.
\end{proof}

We will now look at how jumps are related between two different GDH theories for some fixed edge size $r$. We will see that in general jumps always ``pass up" the subgroup lattice. That is, if $J_{T'} \subseteq J_T$ for GDH theories $T'$ and $T$, then a jump for $T'$ is a jump for $T$. The converse is not true in general. In fact, for any GDH theories $T'$ and $T$ with $J_{T'} \subseteq J_T$ such that the order of $J_T$ is at least three times that of $J_{T'}$ we will show that the set of jumps for $T'$ is not equal to the set of jumps for $T$. The case where $m_T = 2m_{T'}$ is open.

\subsection{Jumps pass up the lattice}

First, we will show that for GDH theories $T$ and $T'$ with $J_{T'} \subseteq J_T$ the set of Tur\'{a}n densities of forbidden families of $T$-graphs is a subset of the set of Tur\'{a}n densities for $T'$.

\begin{theorem}
\label{thmC}
Let $T$ and $T'$ be two GDH theories such that $J_{T'} \subseteq J_{T}$. Then for any family $\mathcal{F}$ of $T$-graphs there exists a family $\mathcal{F}'$ of $T'$-graphs for which $\pi_{T'}(\mathcal{F}') = \pi_{T}(\mathcal{F})$. Moreover, if $\mathcal{F}$ is a finite family, then $\mathcal{F}'$ is also finite.
\end{theorem}

\begin{proof}
For each $F \in \mathcal{F}$ let $F_{T'}$ be the set of all $T'$-graphs that have exactly one edge contained in every edge of $F$. That is, since $J_{T'} \subseteq J_{T}$, then there are $\frac{m_T}{m_{T'}}$ possible $T'$ edges contained within one $T$ edge. So $F_{T'}$ is a finite set with at most $\left(\frac{m_T}{m_{T'}}\right)^{e_T(F)}$ members. Let \[\mathcal{F}' = \bigcup_{F \in \mathcal{F}} F_{T'}.\] Then $\mathcal{F}'$ is a family of $T'$-graphs. Moreover, $\mathcal{F}'$ is finite if $\mathcal{F}$ is finite. We want to show that $\pi_{T'}(\mathcal{F}') = \pi_{T}(\mathcal{F})$.

First, let $\{G_n'\}$ be an extremal $\mathcal{F}'$-free sequence of $T'$-graphs. For each $G_n'$ let $G_n$ be the $T$-graph constructed by replacing each $T'$-edge of $G_n'$ with its containing $T$-edge (multiple $T'$-edges could correspond to the same $T$-edge but each $T$-edge can only be added once).

The sequence $\{G_n\}$ is $\mathcal{F}$-free since otherwise some $G_n$ contains some $F \in \mathcal{F}$ which means that $G_n'$ must have contained at least one member of $F_{T'}$. Therefore, \[\pi_{T}(\mathcal{F}) \geq \lim_{n \rightarrow \infty} d_T(G_n) \geq \lim_{n \rightarrow \infty} \frac{\frac{m_{T'}}{m_T} e_{T'}(G_n')}{\frac{r!}{m_T}{n \choose r}} = \lim_{n \rightarrow \infty} \frac{ex_{T'}(n,\mathcal{F}')}{\frac{r!}{m_{T'}}{n \choose r}} = \pi_{T'}(\mathcal{F}').\]

Conversely, now let $\{G_n\}$ be an extremal $\mathcal{F}$-free sequence of $T$-graphs. For each $G_n$ construct a $T'$-graph $G_n'$ by replacing each $T$-edge with all $\frac{m_T}{m_{T'}}$ $T'$-edges contained in it. The sequence $\{G_n'\}$ is $\mathcal{F}'$-free with $\frac{m_T}{m_{T'}}\text{ex}_T(n,\mathcal{F})$ edges. Therefore, \[\pi_{T'}(\mathcal{F}') \geq \lim_{n \rightarrow \infty} \frac{\frac{m_T}{m_{T'}}\text{ex}_T(n,\mathcal{F})}{\frac{r!}{m_{T'}} {n \choose r}} = \pi_T(\mathcal{F}).\] So $\pi_{T'}(\mathcal{F}') = \pi_{T}(\mathcal{F})$.
\end{proof}

The converse of Theorem~\ref{thmC} is false in general. For example, the permutation subgroup for the theory $T'$ of ($2 \rightarrow 1$)-uniform directed hypergraphs is a subgroup of the permutation group for the theory $T$ of undirected $3$-graphs, $S_3$. The extremal number for the directed hypergraph is $F = \{ab \rightarrow c, cd \rightarrow e\}$ (see Figure~\ref{F}) is \[ex_{T'}(n,F) = \left\lfloor \frac{n}{3} \right\rfloor {\left\lceil \frac{2n}{3} \right\rceil \choose 2}\] as shown in \cite{langlois2010}. Therefore, the Tur\'{a}n density is $\pi_{T'}(F) = \frac{4}{27}$. However, it is well-known that no Tur\'{a}n densities exist for $3$-graphs in the interval $\left(0,\frac{6}{27}\right)$.

\begin{figure}
  \centering
      \begin{tikzpicture}
			\filldraw [black] (-2,2) circle (1pt);
			\filldraw [black] (-2,0) circle (1pt);
			\draw[thick] (-2,2) -- (-2,0);
			\filldraw [black] (0,1) circle (1pt);
			\draw[thick, ->] (-2,1) -- (0,1);
			\filldraw [black] (0,-1) circle (1pt);
			\draw[thick] (0,1) -- (0,-1);
			\draw[thick,->] (0,0) -- (2,0);
			\filldraw [black] (2,0) circle (1pt);
		\end{tikzpicture}
  \caption{$\pi(F) = \frac{4}{27}$}
  \label{F}
\end{figure}
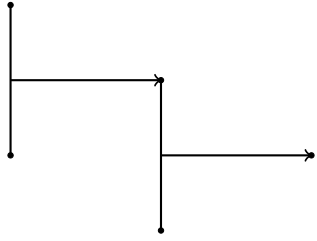

\begin{corollary}
\label{jumpsgoup}
Let $T$ and $T'$ be two GDH theories such that $J_{T'} \subseteq J_{T}$. If $\alpha$ is a jump for $T'$, then it is also a jump for $T$.
\end{corollary}

\begin{proof}
If $\alpha$ is not a jump for $T$, then for any $c > 0$ there exists by Proposition~\ref{thmB} a family $\mathcal{F}$ such that $\alpha < \pi_T(\mathcal{F}) < \alpha + c$. So by Theorem~\ref{thmC} there exists a family $\mathcal{F}'$ of $T'$-graphs with $\alpha < \pi_{T'}(\mathcal{F}') < \alpha + c$. So $\alpha$ is not a jump for $T'$.
\end{proof}

Corollary~\ref{jumpsgoup} immediately implies that all nonjumps found for $r$-uniform undirected hypergraphs must also be non-jumps for any GDH with an $r$-ary relation. However, the converse is not true in general.

\subsection{Jumps do not pass down the lattice}

Roughly speaking, the current best method of demonstrating that a particular $\alpha$ is not a jump for $r$-uniform hypergraphs is to construct a sequence of hypergraphs each with blowup densities that are strictly larger than $\alpha$ but for which any relatively small subgraph has blowup density at most $\alpha$. This method originated in \cite{frankl1984} and generalizes to GDHs as the following definition and lemma demonstrate.

\begin{definition}
Let $\alpha \in [0,1)$. Call $\alpha$ a \emph{demonstrated nonjump} for a GDH theory $T$ if there exists an infinite sequence of GDHs, $\{G_n\}$, such that $b_T(G_n) > \alpha$ for each $G_n$ in the sequence and for any positive integer $l$ there exists a positive integer $n_0$ such that whenever $n \geq n_0$ then any subGDH $H \subseteq G_n$ on $l$ or fewer vertices has blowup density $b_T(H) \leq \alpha$.
\end{definition}

\begin{lemma}
Every demonstrated nonjump is a nonjump.
\end{lemma}

\begin{proof}
Suppose not. Assume that $\alpha$ is a demonstrated nonjump but is a jump. Then there exists a finite family of GDHs $\mathcal{F}$ such that $\pi_T(\mathcal{F}) \leq \alpha$ and $b_T(F) > \alpha$ for each $F \in \mathcal{F}$. Let $l$ be the maximum number of vertices over the members of $\mathcal{F}$. Let $n$ be large enough so that any subGDH on $l$ or fewer vertices has blowup density at most $\alpha$. Then some large enough blowup of $G_n$ contains some $F \in \mathcal{F}$ as a subGDH since the blowup density of each $G_n$ tends to something strictly greater than $\alpha$. Let $H$ be the minimal subGDH of $G_n$ for which the corresponding blowup contains this copy of $F$. Since $H$ has at most $l$ vertices, then it has a blowup density at most $\alpha$. Hence, \[b_T(F) \leq b_T(H(t)) \leq b_T(H) \leq \alpha,\] a contradiction.
\end{proof}

We can now show that a demonstrated nonjump for a GDH theory $T$ yields multiple nonjumps of equal and lesser values down the lattice to GDH theories $T'$ for which $J_{T'} \subseteq J_T$.

\begin{theorem}
Let $T$ and $T'$ be GDH theories such that $J_{T'} \subseteq J_{T}$. Let $\alpha$ be a demonstrated nonjump for $T$. Then $\frac{km_{T'}}{m_T} \alpha$ is a demonstrated nonjump for $T'$ for $k = 1, \ldots, \frac{m_T}{m_{T'}}$.
\end{theorem}

\begin{proof}
Let $\alpha$ be a demonstrated nonjump for $T$. Let $\{G_n\}$ be the corresponding infinite sequence of GDHs. Fix some $k \in \{1,\ldots,\frac{m_T}{m_{T'}}\}$. For each $n$ let $G_n'$ be a $T'$-graph constructed from $G_n$ by replacing each $T$-edge with $k$ $T'$-edges in any orientation. Then by Proposition~\ref{containing} we know that \[b_{T'} (G_n') = \frac{km_{T'}}{m_T}b_T(G_n)\] and any $H' \subseteq G_n'$  corresponding to $H \subseteq G_n$ also gives: \[b_{T'} (H') = \frac{km_{T'}}{m_T}b_T(H).\] Therefore, $b_{T'}(G_n') > \frac{km_{T'}}{m_T} \alpha$ for each $n$ and for any positive integer $l$, there exists a $n_0$ such that $b_{T'}(H) \leq \alpha$ for any subGDH $H \subseteq G_n$ for all $n \geq n_0$.
\end{proof}

Constructions of sequences of undirected $r$-graphs which show that $\frac{5r!}{2r^r}$ is a demonstrated nonjump for each $r \geq 3$ were given in \cite{frankl2007}. This gives the following corollary.

\begin{corollary}
Let $T$ be an $r$-ary GDH theory for $r \geq 3$. Then $\frac{5m_Tk}{2r^r}$ is a nonjump for $T$ for $k=1,\ldots,\frac{r!}{m_T}$.
\end{corollary}

This in turn shows that the set of jumps for a theory $T'$ is a proper subset of the set of jumps for $T$ for any $T$ such that $J_{T'} \subseteq J_T$ and $m_T \geq 3m_{T'}$.

\begin{corollary}
Let $T$ and $T'$ be $r$-ary GDH theories such that $J_{T'} \subseteq J_{T}$ and $m_T \geq 3m_{T'}$. Then there exists an $\alpha$ that is a nonjump for $T'$ and a jump for $T$.
\end{corollary}

\begin{proof}
Take $k=1$, then $\frac{5m_{T'}}{2r^r}$ is a nonjump for $T'$. Since $m_T \geq 3m_{T'}$, then $m_T > 2.5 m_{T'}$. So \[\frac{5m_{T'}}{2r^r} < \frac{m_T}{r^r}.\] Therefore, $\frac{5m_{T'}}{2r^r}$ is a jump for $T$ since every $\alpha \in \left[0,\frac{m_T}{r^r}\right)$ is a jump for $T$.
\end{proof}

\section{Continuity and Approximation}

The following two results are direct adaptations of two theorems from \cite{brown1984}. They are both general extremal results related to everything discussed in this paper but did not fit nicely into the other sections. The first result, Continuity, relates extremal numbers of any infinite family of GDHs to the extremal numbers of its finite subfamilies. The second, Approximation, discusses structural aspects of (nearly) extremal sequences for any forbidden family.

\begin{theorem}[Continuity]
Let $\mathcal{F}$ be an infinite family of $T$-graphs. For each $\epsilon > 0$ there exists a finite subfamily $\mathcal{F}_{\epsilon} \subset \mathcal{F}$ such that \[\text{ex}_T\left(n,\mathcal{F}\right) \leq \text{ex}_T\left(n,\mathcal{F}_{\epsilon}\right) < \text{ex}_T\left(n,\mathcal{F}\right) + \epsilon n^r\] for sufficiently large $n$.
\end{theorem}

\begin{proof}
Let $\mathcal{F}$ be the infinite family of GDHs. For each positive integer $k$ let $\mathcal{F}_k$ be the subfamily of $\mathcal{F}$ where each member has at most $k$ vertices. Let \[\gamma_k = \lim_{n \rightarrow \infty} \frac{\text{ex}_T\left(n, \mathcal{F}_k\right)}{\frac{r!}{m_T} {n \choose r}}\] and let \[\gamma = \lim_{n \rightarrow \infty} \frac{\text{ex}_T\left(n, \mathcal{F}\right)}{\frac{r!}{m_T} {n \choose r}}.\] Since $\mathcal{F}_k \subset \mathcal{F}$, then $\{\gamma_k\}_{k=1}^{\infty}$ is a monotone decreasing sequence and $\gamma_k \geq \gamma$ for all $k$.

Assume for some $\epsilon > 0$ that $\gamma_k > \gamma + \epsilon$ for all $k$. Note that \[\frac{\text{ex}_T\left(n,\mathcal{F}_k\right)}{\frac{r!}{m_T} {n \choose r}} \geq \gamma_k\] is true for all $n$. In particular, when $n=k$ there is an $\mathcal{F}_n$-free GDH on $n$ vertices with strictly more than $\left( \gamma + \epsilon \right) \frac{r!}{m_T} {n \choose r}$ edges. Since an $\mathcal{F}_n$-free GDH on $n$ vertices is also necessarily $\mathcal{F}$-free, then this implies that \[\text{ex}_T\left(n, \mathcal{F} \right) > \left(\gamma + \epsilon\right) \frac{r!}{m_T} {n \choose r},\] a contradiction.
\end{proof}

Theorem 6 in \cite{brown1984} is the Approximation Theorem for totally directed $r$-uniform hypergraphs with bounded multiplicity. We will use the following equivalent statement (in the case of multiplicity one) written in terms of Tur\'{a}n densities as a lemma to prove that this approximation result holds for all GDHs.

\begin{lemma}
\label{approx}
Let $\mathcal{F}'$ be a family of forbidden totally directed $r$-graphs ($r$-GDHs under the trivial group), and let $\epsilon>0$. Then there exists some totally directed $r$-graph $G'$ such that every blowup of $G'$ is $\mathcal{F}'$-free and \[\pi(\mathcal{F}') \geq b(G) > \pi(\mathcal{F}') - \epsilon.\]
\end{lemma}

\begin{theorem}[Approximation]
Let $\mathcal{F}$ be a family of forbidden $T$-graphs, and let $\epsilon > 0$, then there exists some $T$-graph $G$ for which all blowups of $G$ are $\mathcal{F}$-free and \[\pi_T(\mathcal{F}) \geq b_T(G) > \pi_T(\mathcal{F}) - \epsilon.\]
\end{theorem}

\begin{proof}
Let $\mathcal{F}'$ be the family of totally directed $r$-graphs as defined in the proof of Theorem~\ref{thmC}. That is, the family of directed hypergraphs for which we know that $\pi(\mathcal{F}') = \pi_T(\mathcal{F})$. We know from the proof of that theorem that any $T$-graph that is the minimal container for an $\mathcal{F}'$-free graph is $\mathcal{F}$-free. By Lemma~\ref{approx} there exists some totally directed $\mathcal{F}'$-free $r$-graph , $G'$, such that \[\pi(\mathcal{F}') \geq b(G) > \pi(\mathcal{F}') - \epsilon.\] By Proposition~\ref{containing} we know that if $G'$ is the minimal containing $T$-graph of $G$, then $b_T(G) \geq b(G')$. Hence, \[\pi_T(\mathcal{F}) \geq b_T(G) \geq b(G') > \pi(\mathcal{F}') - \epsilon = \pi_T(\mathcal{F}) - \epsilon.\]
\end{proof}

\section{Conclusion}

Some questions naturally come up in studying GDHs. Most notably it would be nice to show that the set of jumps for some GDH theory $T'$ is a proper subset of the set of jumps of any theory $T'$ up the lattice including those for which $m_T = 2m_{T'}$. Or on the other hand it would be very interesting to learn that this is not true in certain cases for $r \geq 3$!

\begin{conjecture}
Let $T'$ and $T$ be $r$-ary GDH theories for $r \geq 3$ such that $J_{T'} \subseteq J_T$ and $m_T = 2m_{T'}$. Then there exists some $\alpha \in [0,1)$ for which $\alpha$ is a jump for $T$ but not for $T'$.
\end{conjecture}

It is known by a result in \cite{brown1984} that every $\alpha \in [0,1)$ is a jump for digraphs. Therefore, the conjecture is not true when $r=2$. On a related note, is it always true that when $J_{T'} \subset J_T$, there always exists a family $\mathcal{F}'$ of $T'$-graphs such that $\pi_{T'}(\mathcal{F}')$ is not contained in the set of Tur\'{a}n densities for $T$?

\begin{conjecture}
Let $T'$ and $T$ be theories such that $J_{T'} \subseteq J_T$. Then there exists some family $\mathcal{F}'$ of $T'$-graphs such that $\pi_{T'}(\mathcal{F}')$ is not contained in the set of Tur\'{a}n densities for $T$.
\end{conjecture}

Finally, it would be nice to generalize the definition of a GDH to include other combinatorial structures. For instance we could easily change the current formulation to include multiple relations in order to capture nonuniform GDHs and those with edges that have bounded multiplicity like the structures studied in \cite{brown1984}. We could even allow these theories to contain general statements that relate the different relations. An example of this might be the theory of some kind of GDH with an edge-coloring that behaves in a certain way (at least locally). In another direction we could take away the requirement that all vertices of an edge be distinct to allow for kinds of generalized loops or add a condition that the existence of certain edges preclude the existence of others such as in the oriented cases studied in \cite{leader2010}, \cite{cameron2015}, and \cite{cameron2015deg}.

\bibliography{GDH}
\bibliographystyle{plain}

\end{document}